\documentclass[10pt]{articlefederico2017}

\usepackage[cmyk]{xcolor}
\usepackage{amsthm}
\usepackage{amsmath}
\usepackage{amssymb}
\usepackage{color}
\usepackage{enumerate}
\usepackage{bbm}
\usepackage{geometry}
\usepackage[cmyk]{xcolor}
\definecolor{mycyan}{cmyk}{1,.2,0,0} 
\definecolor{mymagenta}{cmyk}{0,1,0,0}
\definecolor{DarkOrchid}{RGB}{153,50,204}

\usepackage{caption}
\usepackage{tikz}
\usepackage{lscape}
\usepackage[smalltableaux]{ytableau}
\usepackage{pifont}
\usepackage{array,tabularx}
\usepackage{rotating}
\usepackage[colorlinks]{hyperref}
\usepackage{cleveref}

\usetikzlibrary{calc}
\tikzstyle{vertex}=[circle, draw, inner sep=0pt, minimum size=6pt]

\newtheorem{theorem}{Theorem}[section]
\newtheorem{corollary}[theorem]{Corollary}
\newtheorem{proposition}[theorem]{Proposition}
\newtheorem{lemma}[theorem]{Lemma}
\newtheorem{definition}[theorem]{Definition}
\newtheorem{conjecture}[theorem]{Conjecture}

\newcommand\hobox{
\unitlength .23 mm 
\begin{picture}(10,10)(0,0)
\linethickness{0.2mm}
\put(0,-0.4){\line(0,1){10.83}}
\linethickness{0.2mm}
\put(-0.4,0){\line(1,0){10.83}}
\linethickness{0.2mm}
\multiput(10,0.11)(0,1.885){6}{\line(0,1){0.9}}
\linethickness{0.2mm}
\multiput(0.11,10)(1.885,0){6}{\line(1,0){0.9}}
\end{picture}
\,
}

\makeatletter
\newtheorem*{rep@theorem}{\rep@title} \newcommand{\newreptheorem}[2]{%
\newenvironment{rep#1}[1]{%
\def\rep@title{\bf #2 \ref{##1}}%
\begin{rep@theorem} }%
{\end{rep@theorem} } }
\makeatother
\newreptheorem{theorem}{Theorem}
\newreptheorem{conjecture}{Conjecture}

\theoremstyle{definition}
\newtheorem{example}[theorem]{Example}

\newcommand{\Span}{\operatorname{span}}
\newcommand{\aff}{\operatorname{aff}}
\newcommand{\Ehr}{\operatorname{Ehr}}
\newcommand{\vol}{\operatorname{vol}}
\newcommand{\val}{\operatorname{val}_2}

\newcommand{\bv}{\mathbf{v}}
\newcommand{\bu}{\mathbf{u}}
\newcommand{\bo}{\mathbf{0}}

\newcommand{\bx}{\mathbf{x}}
\newcommand{\by}{\mathbf{y}}
\newcommand{\be}{\mathbf{e}}

\newcommand{\ysub}[1]{\resizebox{1.5ex}{!}{\ydiagram{#1}}}
\renewcommand{\phi}{H^*}

\newcommand{\Z}{\mathbb{Z}}
\newcommand{\Q}{\mathbb{Q}}
\newcommand{\R}{\mathbb{R}}
\newcommand{\C}{\mathbb{C}}

\begin{document}

\title{\textsf{The equivariant Ehrhart theory of the permutahedron}}

\author{
    \textsf{Federico Ardila\footnote{\noindent \textsf{San Francisco State University; Universidad de Los Andes; federico@sfsu.edu.}}}
\and
    \textsf{Mariel Supina\footnote{\noindent \textsf{University of California, Berkeley; mariel\_supina@berkeley.edu.}}}
\and
    \textsf{Andr\'es R. Vindas-Mel\'endez\footnote{\noindent \textsf{San Francisco State University; University of Kentucky; andres.vindas@uky.edu.
\newline 
    The authors were supported by NSF Award DMS-1600609 and DMS-1855610 and Simons Fellowship 613384 (FA), the Graduate Fellowships for STEM Diversity (MS), and NSF Graduate Research Fellowship DGE-1247392 (ARVM).}}}}

\date{}
\maketitle

\begin{abstract} 
Equivariant Ehrhart theory enumerates the lattice points in a polytope with respect to a group action. Answering a question of Stapledon, we describe the equivariant Ehrhart theory of the permutahedron, and we prove his Effectiveness Conjecture in this special case.
\end{abstract}

\section{Introduction}

Ehrhart theory measures a polytope $P$ by counting the number $L_P(t)$ of lattice points in its dilations $tP$ for $t \in \Z_{\ge 0}$. The enumeration of lattice points in  polytopes plays an important role in numerous areas of mathematics.
For example, in algebraic geometry, a lattice polytope $P$ corresponds to a projective toric variety $X_P$ and an ample line bundle $L$, whose Hilbert polynomial is precisely the Ehrhart polynomial $L_P(t)$. The Ehrhart polynomial can be expressed in terms of the Todd class of the toric variety $X_P$ \cite{Morelli, Pommersheim}.
In the representation theory of semisimple Lie algebras, the Kostant partition function enumerates the lattice points in a natural family of polytopes. \cite{BaldoniVergne, MeszarosMorales}
These and many other connections have allowed for the combinatorial computation of many quantities of geometric and algebraic interest. They have also inspired numerous questions and provided interesting answers in Ehrhart theory itself. \cite{CoxLittleSchenck, Fulton}

Motivated by mirror symmetry, Stapledon \cite{Stapledon, Stapledonmirror} introduced \emph{equivariant Ehrhart theory} as a refinement of Ehrhart theory that takes into account the symmetries of the polytope $P$.
Let $G$ be a finite group acting linearly on a lattice polytope $P$.
Combinatorially, the goal of equivariant Ehrhart theory is to understand, for each $g \in G$, the lattice point enumerator $L_{P^g}(t)$ of the polytope $P^g \subseteq P$ fixed by $g$.
These quantities can be assembled into a sequence of virtual characters $H^*_0, H^*_1, H^*_2, \ldots$ of $G$, which one wishes to understand representation-theoretically.
Geometrically, these virtual characters arise naturally when one studies the action of $G$ on the cohomology of a $G$-invariant hypersurface in the toric variety associated to $P$. Stapledon showed that if $(X_P,L)$ admits a non-degenerate $G$-invariant hypersurface, then these virtual characters are effective and polynomial; in particular, they correspond to an actual representation of $G$. His Effectiveness Conjecture \cite[Conjecture 12.1]{Stapledon} states that the converse statement also holds.

To date, few examples of equivariant Ehrhart theory are understood. 
Stapledon computed it for regular simplices, hypercubes, and centrally symmetric polytopes.
If $\Delta$ is the Coxeter fan associated to a root system and $P$ is the convex hull of the primitive integer vectors of the rays of $\Delta$, he used the equivariant Ehrhart theory of $P$ to recover Procesi, Dolgachev--Lunts, and Stembridge's formula \cite{Procesi, DolgachevLunts, Stembridge} for the character of the action of the Weyl group on the cohomology of the toric variety $X_\Delta$.
In \cite{Stapledon},  Stapledon asked for the computation of the next natural example: the permutahedron under the action of the symmetric group. 
The corresponding toric variety is the permutahedral variety, which is the subject of great interest. For example, Huh used it in his Ph.D. thesis \cite{Huhthesis} to prove Rota's conjecture on the log-concavity of the coefficients of chromatic polynomials. In algebraic geometry, it arises as the Losev-Manin moduli space of curves \cite{LosevManin}.

The goal of this paper is to answer Stapledon's question: we compute the equivariant Ehrhart theory of the permutahedron and verify his Effectiveness Conjecture in this special case. Our proofs combine tools from discrete geometry, combinatorics, number theory, algebraic geometry, and representation theory. 

A significant new challenge that arises is that the fixed polytopes of the permutahedron are not integral. Thus the equivariant Ehrhart theory of the permutahedron requires surprisingly subtle arithmetic considerations -- which are absent from the ordinary Ehrhart theory of lattice polytopes -- as the following theorem illustrates.

\subsection{The Ehrhart quasipolynomials of the fixed polytopes of the permutahedron}\label{sec:quasi}

We consider the action of the symmetric group $S_n$ on the $(n-1)$-dimensional permutahedron $\Pi_n$.
For each permutation $\sigma\in S_n$, we define the \emph{fixed polytope} $\Pi_n^\sigma \subseteq \Pi_n$ to be the subset of the permutahedron $\Pi_n$ fixed by $\sigma$. 
Our first main result is a combinatorial formula for its lattice point enumerator $L_{\Pi_n^\sigma}(t) := |t\Pi_n^\sigma\cap \Z^n|$:

\begin{theorem}\label{thm:main theorem}
Let $\sigma$ be a permutation of $[n]:=\{1,2,\ldots,n\}$ and let $\lambda=(\ell_1, \ldots, \ell_m)$ be the partition of $[n]$ given by the lengths of the cycles of $\sigma$. Say a set partition $\pi = \{B_1,\dots,B_k\}$ of $[m]$ is \emph{$\lambda$-compatible} if for each block $B_i$, either $\ell_j$ is odd for some $j \in B_i$, or the minimum $2$-valuation among $\{\ell_j : j \in B_i\}$ is attained an even number of times.
Also write
    \begin{equation}\label{eq:v_pi}
        v_\pi = \prod_{i=1}^k\left(
        \gcd(\ell_j:j\in B_i)\cdot
        \Big(\sum_{j\in B_i}\ell_j\Big)^{|B_i|-2}\right).
    \end{equation}
Then the Ehrhart quasipolynomial of the fixed polytope $\Pi_n^\sigma$ is
    \begin{equation*}
        L_{\Pi_n^\sigma}(t) = 
            \begin{cases}
                \displaystyle \quad \,\, \sum_{\pi \vDash[m]}v_\pi \cdot t^{m-|\pi|} &\text{if $t$ is even}
                \\
                \displaystyle\sum_{\substack{\pi \vDash[m]\\\lambda-\text{compatible}}}  \hspace{-.5cm} v_\pi \cdot t^{m-|\pi|} &\text{if $t$ is odd}
            \end{cases}.
    \end{equation*}
\end{theorem}

As the compatibility condition of Theorem \ref{thm:main theorem} and Example \ref{ex:line} illustrate, there is some delicate number theory at play, even for the line segments that arise in the equivariant Ehrhart theory of the permutahedron.

\subsection{Equivariant Ehrhart theory}\label{sec:equivEhr}

\Cref{thm:main theorem} fits into the framework of equivariant Ehrhart theory, as we now explain.

Let $G$ be a finite group acting on $\Z^n$ and $P\subseteq \R^n$ be a $d$-dimensional lattice polytope that is invariant under the action of $G$. Let $M$ be the sublattice of $\Z^n$ obtained by translating the affine span of $P$ to the origin, and consider the induced representation $\rho: G \rightarrow GL(M)$.
We then obtain a family of permutation representations by looking at how $\rho$ permutes the lattice points inside the dilations of $P$.
Let $\chi_{tP}: G \rightarrow \C$ denote the permutation character associated to the action of $G$ on the lattice points in the $t^{\text{th}}$ dilate of $P$. We have
\[
\chi_{tP}(g) =   L_{P^g}(t) 
\]
where $P^g$ is the polytope of points in $P$ fixed by $g$ and $L_{P^g}(t)$ is its lattice point enumerator.

The permutation characters $\chi_{tP}$ live in the ring $R(G)$ of virtual characters of $G$, which are the integer combinations of the irreducible characters of $G$. The positive integer combinations are called \emph{effective}; they are the characters of representations of $G$.

Stapledon encoded the characters $\chi_{tP}$ in a power series $\phi[z] \in R(G)[[z]]$ given by 
\begin{equation}\label{eq:phi_equation}
    \sum_{t\geq 0}\chi_{tP}(g)z^t= \frac{\phi[z](g)}{(1-z)\det(I- g \cdot z)}.
\end{equation}
We say that $\phi[z] =: \sum_{i\ge0}\phi_iz^i $ is $\emph{effective}$ if each virtual character $\phi_i$ is a character.
Stapledon denoted this series $\varphi[t]$, but we denote it $\phi[z]$ and call it the \emph{equivariant $\phi$-series} because for the identity element, $\phi[z](e) = h^*[z]$ is the well-studied $h$*-polynomial of $P$.

The main open problem in equivariant Ehrhart theory is to characterize when $\phi[z]$ is effective, and Stapledon offered the following conjecture.

\begin{conjecture}[{\cite[Effectiveness Conjecture~12.1]{Stapledon}}]\label{conj:Stapledon}
Let $P$ be a lattice polytope invariant under the action of a group $G$.
The following conditions are equivalent.

\begin{enumerate}[(i)]
	\item The toric variety of $P$ admits a $G$-invariant non-degenerate hypersurface.
	\item The equivariant $\phi$-series of $P$ is effective.
	\item The equivariant $\phi$-series of $P$ is a polynomial.
	\end{enumerate}
\end{conjecture}

Our second main result is the following.

\begin{theorem}\label{thm:phi polynomial}
Stapledon's Effectiveness Conjecture holds for the permutahedron under the action of the symmetric group.
\end{theorem}

Finally, in \Cref{prop:conjectures hold} we verify the three remaining conjectures of Stapledon in this case.

\subsection{Examples}\label{sec:examples}

\begin{example}
Let us illustrate these results for the permutahedron $\Pi_4$ and the permutation $\sigma = (12)(3)(4)$ which has cycle type $\lambda = (2,1,1)$, illustrated in \Cref{fig:example}. The fixed polytope $\Pi_4^{(12)}$ is a half-integral hexagon, and one may verify manually that 
\[
L_{\Pi_4^{(1 2)}}(t) = 
\begin{cases}
    4t^2+3t+1 &\text{if } t \text{ is even}\\
    4t^2+2t &\text{if } t \text{ is odd},
\end{cases}
\qquad \qquad
\phi[z](12)  = 1+4z+11z^2-2z^3+\frac{4z^4}{1+z}.
\]
Since the $\phi$-series of $\Pi_4$ is not polynomial when evaluated at $(12)$, Stapledon's \Cref{conj:Stapledon} predicts that it is also not effective, and that the permutahedral variety $X_{\Pi_4}$ does not admit an $S_4$-invariant non-degenerate hypersurface. We verify this in \Cref{sec:phi}.

\begin{figure}[h]
    \centering
    \includegraphics[scale=.5]{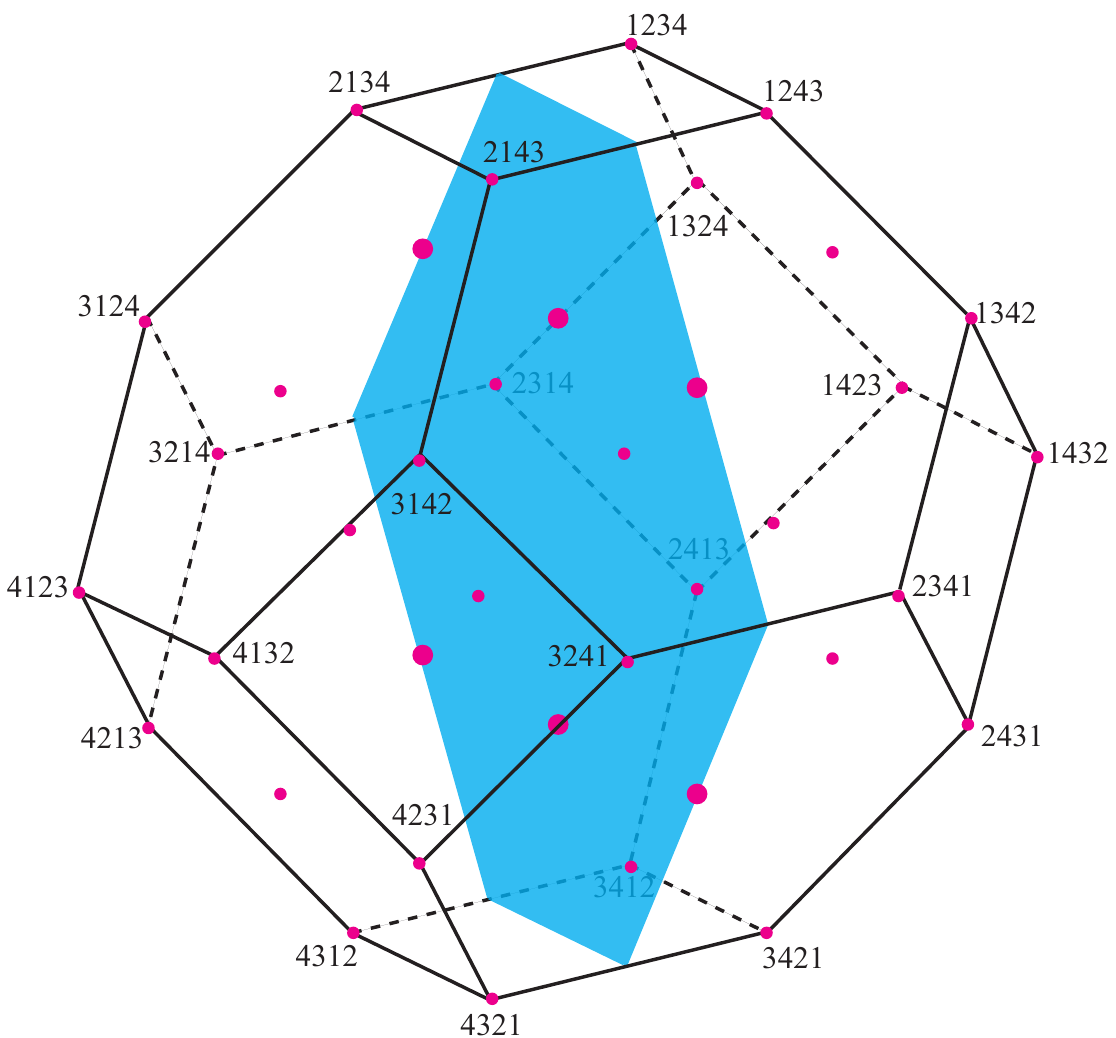} 
    \caption{The fixed polytope $\Pi_4^{(12)}$ is a half-integral hexagon containing 6 lattice points.}
    \label{fig:example}
\end{figure}

The equivariant Ehrhart quasipolynomials and $\phi$-series of $\Pi_3$ and $\Pi_4$ are shown in Tables \ref{tab:n=3} and \ref{tab:n=4}.
\end{example}

\begin{example} \label{ex:line}
Further subtleties already arise in the simple case when $\Pi_n^\sigma$ is a segment; this happens when $\sigma$ has only two cycles of lengths $\ell_1$ and $\ell_2$. 
For even $t$, we simply have
\[
L_{\Pi_n^\sigma}(t) = 
\gcd(\ell_1,\ell_2)t + 1. 
\]
However, for odd $t$ we have
\[
L_{\Pi_n^\sigma}(t) = 
\begin{cases}
\gcd(\ell_1,\ell_2)t + 1 & \textrm{if $\ell_1$ and $\ell_2$ are both odd}, \\
\gcd(\ell_1,\ell_2)t & \textrm{if $\ell_1$ and $\ell_2$ have different parity}, \\
\gcd(\ell_1,\ell_2)t & \textrm{if $\ell_1$ and $\ell_2$ are both even and they have the same $2$-valuation}, \\
0 & \textrm{if $\ell_1$ and $\ell_2$ are both even and they have different $2$-valuations,}
\end{cases}
\]
We invite the reader to verify that this formula follows from \Cref{thm:main theorem}.
\end{example}

\subsection{Organization}

In \Cref{sec:background} we introduce some background on Ehrhart theory and zonotopes.
In \Cref{sec:lattice case} we compute the Ehrhart polynomial of the fixed polytope $\Pi_n^\sigma$ when it is a lattice polytope, and in \Cref{sec:Ehrhart quasipolynomial} we compute its Ehrhart quasipolynomial in general, proving \Cref{thm:main theorem}. 
In \Cref{sec:phi} we compute the equivariant $\phi$-series $\phi[z]$ for permutahedra and we verify Stapledon's four conjectures on equivariant Ehrhart theory in this special case; most importantly, his Effectiveness Conjecture (\Cref{thm:phi polynomial}).

\section{Preliminaries}
\label{sec:background}

\subsection{Ehrhart quasipolynomials}

Let $P$ be a convex polytope in $\R^n$.
We say that $P$ is a \emph{rational polytope} if all of its vertices are in $\Q^n$.
The \emph{lattice point enumerator} of $P$ is the function $L_P:\Z_{\ge1}\to\Z_{\ge0}$ given by
    $$
    L_P(t) := |tP\cap\Z^n|;
    $$
that is, $L_P(t)$ is the number of integer points in the $t^{\text{th}}$ dilate of $P$.
A function $f:\Z \rightarrow \R$ is a \emph{quasipolynomial} if there exists a period $d$ and polynomials $f_0, f_1, \ldots, f_{d-1}$ such that $f(n) = f_i(n)$ whenever $n \equiv i$ (mod $d$).

\begin{theorem}[Ehrhart's Theorem \cite{Ehrhart1962}]\label{thm:ehrharts theorem rational}
If $P$ is a rational polytope, then $L_P(t) := |tP\cap\Z^n|$ is a quasipolynomial in $t$. Its degree is $\dim P$ and  its period divides the least common multiple of the denominators of the coordinates of the vertices of $P$. 
\end{theorem}

\subsection{Zonotopes}

Let $V$ be a finite set of vectors in $\R^n$.
The \emph{zonotope} generated by $V$, denoted $Z(V)$, is defined to be the Minkowski sum of the line segments connecting the origin to $\bv$ for each $\bv\in V$.
We will also adapt the same notation to refer to any translation of $Z(V)$, that is, the Minkowski sum of any collection of line segments whose direction vectors are the elements of $V$.
Zonotopes have a combinatorial decomposition that is useful when calculating volumes and counting lattice points.  
The following result is due to Shephard.

\begin{proposition}[{\cite[Theorem~54]{Shephard1974}}]
A zonotope $Z(V)$ can be subdivided into half-open parallelotopes that are in bijection with the linearly independent subsets of $V$.
\end{proposition}

A linearly independent subset $S\subseteq V$ corresponds under this bijection to the half-open parallelotope 
    $$
    \hobox S :=\sum_{\bv\in S} [\bo,\bv).
    $$
Stanley gave a combinatorial formula for the Ehrhart polynomial of a lattice zonotope.

\begin{theorem}[{\cite[Theorem~2.2]{Stanleyzonotope}}] \label{thm:Stanley}
Let $Z(V)$ be a lattice zonotope generated by $V$. Then
    \begin{equation}\label{eq:stanley_formula}
        L_{Z(V)}(t) = \sum_{\substack{S\subseteq V\\\text{lin. indep.}}}\vol(\hobox S)\cdot t^{|S|}.
    \end{equation}
    \end{theorem}

\noindent In the statement above and throughout the paper, volumes are normalized so that any primitive lattice parallelotope has volume $1$.

\subsection{Fixed polytopes of the permutahedron}

The symmetric group $S_n$ acts on $\R^n$ by permuting coordinates of points.
The \emph{permutahedron} $\Pi_n$ is the convex hull of the $S_n$-orbit of the point $(1,2,\dots,n)\in\R^n$; that is, of the $n!$ permutations of $[n]$.
 
Let $\sigma \in S_n$ be a permutation with cycles $\sigma_1, \ldots, \sigma_m$; their lengths form a partition $\lambda = (\ell_1,\dots,\ell_m)$ of $n$.
For each cycle $\sigma_k$ of $\sigma$, let $\be_{\sigma_k}=\sum_{i\in\sigma_k}\be_i$.
The \emph{fixed polytope} $\Pi_n^\sigma$ is defined to be the polytope consisting of all points in $\Pi_n$ that are fixed under the action of $\sigma$.
We will use a few results from \cite{ASV}, which we now summarize.

\begin{theorem}[{\cite[Theorems~1.2 and 2.12]{ASV}}]\label{thm:vol}
The fixed polytope $\Pi_n^\sigma$ has the following zonotope description:
    \begin{equation}\label{eq:fixed_polytope_zonotope_desc}
        \Pi_n^\sigma = \sum_{1\le i<j\le m}[\ell_i \be_{\sigma_j},\ell_j \be_{\sigma_i}] + \sum_{k=1}^m \frac{\ell_k+1}{2}\be_{\sigma_k}.
    \end{equation}
Its normalized volume is
\begin{equation} \label{eq:volume}
    \vol \Pi_n^\sigma = n^{m-2}\gcd(\ell_1,\dots,\ell_m).
\end{equation}
\end{theorem}

\begin{corollary}\label{cor:oddparts}
The fixed polytope $\Pi_n^\sigma$ is integral or half-integral. 
It is a lattice polytope if and only if all cycles of $\sigma$ have odd length.
\end{corollary}
\begin{proof}
From \eqref{eq:fixed_polytope_zonotope_desc} and from the fact that all of the $\be_{\sigma_i}$ in \eqref{eq:fixed_polytope_zonotope_desc} are linearly independent, we can see that all the vertices of $\Pi_n^\sigma$ will be in the integer lattice if and only if $\ell_i+1$ is even for all $i$.
\end{proof}

Equation \eqref{eq:fixed_polytope_zonotope_desc} also shows that $\Pi_n^\sigma$ is a rational translation of the zonotope $Z(V)$ where \[V=\{\ell_i\be_{\sigma_j}-\ell_j\be_{\sigma_i}:1\le i<j\le m\}.\]
The following result characterizes the linearly independent subsets of $V$.

\begin{lemma}[{\cite[Lemma~3.2]{ASV}}]\label{lem:lin_indep_sets}
The linearly independent subsets of $V$ are in bijection with forests with vertex set $[m]$, where the vector $\ell_i\be_{\sigma_j}-\ell_j\be_{\sigma_i}$ corresponds to the edge connecting vertices $i$ and $j$.
\end{lemma}

In light of this lemma, the fixed polytope $\Pi_n^\sigma$ gets subdivided into half-open parallelotopes $\hobox_F$ of the form
    \begin{equation} \label{eq:boxF}
        \hobox_F = \sum_{\{i,j\}\in E(F)} [\ell_i\be_{\sigma_j},\ell_j\be_{\sigma_i}) + \sum_{k=1}^m \frac{\ell_k+1}{2}\be_{\sigma_k} + \bv_F,  \qquad \qquad \bv_F \in \Z^n
    \end{equation}
for each forest $F$ with vertex set $[m]$. 
When $F$ is a tree $T$ we have that
    $$
    \vol (\hobox_T) = \left(\prod_{i=1}^m \ell_i^{\deg_T(i)-1}\right)\gcd(\ell_1,\dots,\ell_m).
    $$
by \cite[Lemma~3.3]{ASV}. 
For a general forest $F$, the parallelotopes $\hobox_T$ corresponding to each connected component $T$ of $F$ live in orthogonal subspaces, so
    \begin{equation}\label{eq:volume_forest}
        \vol(\hobox_F) = \Big( \prod_{j=1}^m\ell_j^{\deg_F(j)-1} \Big)\bigg( \prod_{\substack{\text{conn. comp.}\\ T\text{ of } F}} \gcd(\ell_j:j\in\mathrm{vert}(T)) \bigg).
    \end{equation}

\section{The Ehrhart polynomial of the fixed polytope: the lattice case}
\label{sec:lattice case}

Suppose that $\lambda=(\ell_1,\dots,\ell_m)$ is a partition of $n$ into odd parts and that $\sigma\in S_n$ has cycle type $\lambda$.
Then \Cref{cor:oddparts} says that $\Pi_n^\sigma$ is a lattice zonotope, and hence we can use \eqref{eq:stanley_formula} to write a combinatorial expression for its Ehrhart polynomial.
Recall the definition of $v_\pi$ in \eqref{eq:v_pi}.

\begin{theorem}\label{thm:ehrhart_poly_lattice}
Let $\sigma\in S_n$ have cycle type $\lambda=(\ell_1,\dots,\ell_m)$, where $\ell_i$ is odd for all $i$.
Then
    \begin{equation*}
        L_{\Pi_n^\sigma}(t) = \sum_{\pi \vDash[m]}v_\pi \cdot t^{m-|\pi|}
    \end{equation*}
summing over all partitions $\pi = \{B_1,\dots,B_k\}$ of $[m]$.
\end{theorem}

\begin{proof}
Combining \Cref{thm:Stanley} with \eqref{eq:volume_forest} gives us the following formula for the Ehrhart polynomial of $\Pi_n^\sigma$:
    \begin{equation}\label{eq:ehrhart_sumoverforests}
        L_{\Pi_n^\sigma}(t) = \sum_{\substack{\text{Forests } F \\\text{on } [m]}} \Big( \prod_{j=1}^m \ell_j^{\deg_F(j)-1} \Big) \cdot \Big(\prod_{\substack{\text{conn. comp.}\\ T\text{ of } F}} \gcd(\ell_j:j\in\mathrm{vert}(T)) \Big) t^{|E(F)|}.
    \end{equation}
We can construct a forest with vertex set $[m]$ by first partitioning $[m]$ into nonempty sets $\{B_1,\dots,B_k\}$ and then choosing a tree with vertex set $B_j$ for each $j$.
The number of edges in such a forest is $m-k$.
Using these observations, we can rewrite \eqref{eq:ehrhart_sumoverforests} as
    \begin{equation*}
        L_{\Pi_n^\sigma}(t) = \sum_{\{B_1,\dots,B_k\}\vDash[m]} \Big(\prod_{i=1}^k\gcd(\ell_j:j\in B_i) \Big)\cdot\Big( \sum_{\substack{\text{Forests }F\\\text{inducing}\\
        \{B_1,\dots,B_k\}}}\prod_{j=1}^m\ell_j^{\deg_F(j)-1} \Big) t^{m-k}.
    \end{equation*}
To complete the proof, it remains to show that for a given partition $\pi=\{B_1,\dots,B_k\}$ of $[m]$, the following identity holds:
    \begin{equation}\label{eq:forest_sum}
        \sum_{\substack{\text{Forests }F\\\text{inducing}\\
        \{B_1,\dots,B_k\}}}\prod_{j=1}^m\ell_j^{\deg_F(j)-1} = \prod_{i=1}^k\Big(\sum_{j\in B_i}\ell_j\Big)^{|B_i|-2}.
    \end{equation}    
This follows from the following identity, found in \cite[Lemma~3.4]{ASV}:
    \begin{equation}\label{eq:tree_sum}
        \sum_{T \text{ tree on }[m]}\prod_{i=1}^m x_j^{\deg_T(j)-1} = (x_1+\dots+x_m)^{m-2}.
    \end{equation}
Using \eqref{eq:tree_sum} we obtain 
    \begin{align*}
        \sum_{\substack{\text{Forests }F\\\text{inducing}\\
        \{B_1,\dots,B_k\}}}\prod_{j=1}^m\ell_j^{\deg_F(j)-1}
        &= \sum_{\substack{\text{Forests }F\\\text{inducing}\\
        \{B_1,\dots,B_k\}}} \prod_{i=1}^k \prod_{j\in B_i}\ell_j^{\deg_F(j)-1}\\
        &= \prod_{i=1}^k \bigg(\sum_{\substack{\text{trees }T\\\text{on } B_i}}\prod_{j\in B_i}\ell_j^{\deg_T(j)-1}\bigg)\\
        &= \prod_{i=1}^k \Big( \sum_{j\in B_i}\ell_j \Big)^{|B_i|-2}
    \end{align*}
as desired.
\end{proof}

\section{The Ehrhart quasipolynomial of the fixed polytope: the general case}
\label{sec:Ehrhart quasipolynomial}

In general, $\Pi_n^\sigma$ is a half-integral polytope.
This means that instead of an Ehrhart polynomial, it has an Ehrhart quasipolynomial with period at most $2$.
As in the lattice case from \Cref{sec:lattice case}, we can decompose $\Pi_n^\sigma$ into half-open parallelotopes.
However, there is a new feature that does not arise in the lattice case: some of the parallelotopes in this decomposition may not contain any lattice points.

\begin{figure}[h]
    \begin{center}
       \scalebox{.7}{ \begin{tikzpicture}
[	back/.style={loosely dotted, thick},
	edge/.style={color=gray!95!black, thick},
	facet/.style={fill=gray!95!black,fill opacity=0.400000},
	mydot/.style={circle,inner sep=1.5pt,circle,draw=black!25!black,fill=white!75!white,thick,anchor=base},
  	vertex/.style={inner sep=1.5pt,circle,draw=black!25!black,fill=black!75!black,thick,anchor=base},
  	point/.style={inner sep=2pt,circle,draw=mymagenta!25!mymagenta,fill=mymagenta!75!mymagenta,thick,anchor=base}	]
  	\begin{scope} 

\coordinate  (A) at (0,2);
\coordinate  (B) at ({sqrt(2)},2); 
\coordinate  (C) at ({sqrt(8)},0);
\coordinate  (D) at ({sqrt(2)},-2);
\coordinate  (E) at (0,-2); 
\coordinate  (F) at (-{sqrt(2)},0);

\draw[top color=mycyan, bottom color=mycyan, fill opacity=1] (A) to (B) to (C) to (D) to (E) to (F) to (A); 

\node[point] at ({-sqrt(2/4)},1)  {};
\node[point] at ({3*sqrt(2/4)},1)  {};
\node[point] at ({sqrt(2/4)},1)  {};
\node[point] at ({3*sqrt(2/4)},1)  {};
\node[point] at ({3*sqrt(2/4)},-1)  {};
\node[point] at ({sqrt(2/4)},-1)  {};
\node[point] at ({-sqrt(2/4)},-1)  {};
\end{scope}


\begin{scope}
\coordinate (A) at (6,2);
\coordinate (B) at ({6+sqrt(2)},2); 
\coordinate (C) at ({6+sqrt(8)},0);
\coordinate (D) at ({6+sqrt(2)},-2);
\coordinate (E) at (6,-2); 
\coordinate (F) at ({6-sqrt(2)},0);
\coordinate (G) at ({6+sqrt(2)},0);
\draw[white, top color=mycyan, middle color=mycyan, bottom color=mycyan, fill opacity=1] (F) to (A) to (G) to (E) to (F);
\draw(F) to (A) to (G);
\draw[dashed] (G) to (E) to (F);
\coordinate[mydot] (F) at ({6-sqrt(2)},0);
\coordinate[mydot] (G) at ({6+sqrt(2)},0);

\coordinate[vertex] (H) at ({4.5+sqrt(2)},-2.4);

\coordinate[mydot] (I) at (6.2,-2.5);
\coordinate (J) at ({6.2+sqrt(2)},-2.5);
\draw (I) to (J);

\coordinate[mydot] (K) at (5.7,-2);
\coordinate (L) at ({5.7-sqrt(2)},0);
\draw (K) to (L);

\coordinate (M) at (6.3,2);
\coordinate (N) at ({6.3+sqrt(2)},2); 
\coordinate (O) at ({6.3+sqrt(8)},0);
\coordinate (P) at ({6.3+sqrt(2)},0); 
\draw[white, top color=mycyan, middle color=mycyan, bottom color=mycyan, fill opacity=1] (M) to (N) to (O) to (P) to (M);
\draw (M) to (N) to (O);
\draw[dashed] (M) to (P) to (O);
\coordinate[mydot] (M) at (6.3,2);
\coordinate[mydot] (O) at ({6.3+sqrt(8)},0);

\coordinate[mydot] (Q) at (6.2,-2.1);
\coordinate (R) at ({6.2+sqrt(2)},-0.1);
\draw (Q) to (R);

\coordinate(R) at at ({6.3+sqrt(2)},-0.3); 
\coordinate (S) at ({6.3+sqrt(8)},-0.3);
\coordinate (T) at at ({6.3+sqrt(2)},-2.3); 
\coordinate (U) at (6.3,-2.3);
\draw[white, top color=mycyan, middle color=mycyan, bottom color=mycyan, fill opacity=1] (R) to (S) to (T) to (U) to (R);
\draw (R) to (S) to (T);
\draw[dashed] (T) to (U) to (R);
\coordinate[mydot] (R) at at ({6.3+sqrt(2)},-0.3); 
\coordinate[mydot] (T) at at ({6.3+sqrt(2)},-2.3); 


\node[point] at ({5.7-sqrt(2/4)},-1)  {};
\node[point] at ({6-sqrt(2/4)},1)  {};
\node[point] at ({6+sqrt(2/4)},1)  {};
\node[point] at ({6.3+sqrt(2/4)+sqrt(8/4)},1)  {};
\node[point] at ({6.2+sqrt(2/4)},-1.1)  {};
\node[point] at ({6.3+sqrt(2/4)+sqrt(8/4)},-1.3)  {};

\end{scope}
\end{tikzpicture}}
    \end{center}
    \caption{Decomposition of the fixed polytope $\Pi_4^{(12)}$ into half-open parallelotopes.}
    \label{fig:half_integral_hexagon} 
\end{figure}
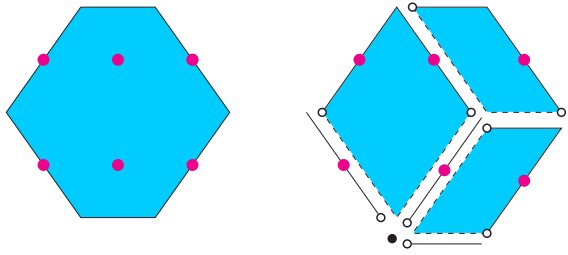

\begin{example}\label{ex:quasipolynomial}
The fixed polytope $\Pi_4^{(1 2)}$ of \Cref{fig:example}, which corresponds to the cycle type $\lambda = (2,1,1)$, is
$$
\Pi_4^{(12)} = [2\be_{3},\be_{12}] + [2\be_{4},\be_{12}] + [\be_{4},\be_{3}]+\frac{3}{2}\be_{12}+\be_{3}+\be_{4}.
 $$
\Cref{fig:half_integral_hexagon} shows its decomposition into parallelograms indexed by the forests on vertex set $\{12, 3, 4\}$.
The three trees give parallelograms with volumes $2, 1, 1$ that contain $2,1,1$ lattice points, respectively. 
The three forests with one edge give segments of volumes $1, 1, 1$ and  $1,1,0$ lattice points, respectively. 
The empty forest gives a point of volume $1$ and $0$ lattice points. 
Hence the Ehrhart quasipolynomial of $\Pi_4^{(12)}$ is
    $$
    L_{\Pi_4^{(1 2)}}(t) = \begin{cases}
    (2+1+1)t^2+(1+1+1)t+1 &\text{if } t \text{ is even}\\
    (2+1+1)t^2+(1+1+0)t+0 &\text{if } t \text{ is odd}
    \end{cases}.
    $$
\end{example}

Following the reasoning of \Cref{ex:quasipolynomial}, we will find the Ehrhart quasipolynomial of $\Pi_n^\sigma$ by examining its decomposition into half-open parallelotopes.
In order to find the number of lattice points in each parallelotope $\hobox_F$, the following observation is crucial.

\begin{lemma}\label{lem:num_lattice_points} \cite{ArdilaBeckMcWhirter, McWhirter}
If $\hobox$ is a half-open lattice parallelotope in $\Z^n$ and $\bv \in \Q^n$, the number of lattice points in $\hobox + \bv$ is
\[
|(\hobox + \bv) \cap \Z^n| = 
\begin{cases}
	\vol(\hobox) & \text{if the affine span of } \hobox+\bv \text{ intersects the lattice } \Z^n \\
	0 & \text{otherwise}
\end{cases}.
\]
\end{lemma}

We now apply \Cref{lem:num_lattice_points} to the parallelotopes $\hobox_F$. Surprisingly,  whether $\aff(\hobox_F)$ contains lattice points does not depend on the forest $F$, but only on the set partition $\pi$ of the vertex set $[m]$ induced by the connected components of $F$. To make this precise we need a definition.
Recall that the $2$\emph{-valuation} of a positive integer is the largest power of $2$ dividing that integer; for example, $\val(24)=3$.

\begin{definition}\label{def:compatibility}
Let $\lambda=(\ell_1,\dots,\ell_m)$ be a partition of the integer $n$.
A set partition $\pi=\{B_1,\dots,B_k\}$ of $[m]$ is called $\lambda$-\emph{compatible} if for each block $B_i\in \pi$, at least one of the following conditions holds:
    \begin{enumerate}[(i)]
        \item $\ell_j$ is odd for some $j\in B_i$, or
        \item the minimum $2$-valuation among $\{\ell_j:j\in B_i\}$ occurs an even number of times.
    \end{enumerate}
\end{definition}

\begin{example}
Let $\lambda = (\ell_1,\ell_2,\ell_3)$ and $\val(\ell_i) = v_i$ for $i=1,2,3$, and assume that $v_1\geq v_2 \geq v_3$.
\Cref{tab:lambda_compatibility} shows which partitions of $[3]$ are $\lambda$-compatible depending on $\val(\lambda)$.
    \begin{table}[h]
        \begin{center}
             \begin{tabular}{| c | c c c c c|} 
 \hline
  & $123$ & $12\vert3$ & $13\vert2$ & $23\vert1$ & $1\vert2\vert3$ \\ [0.5ex] 
 \hline\hline 
 $v_1=v_2=v_3=0$& $\bullet$ & $\bullet$ & $\bullet$ & $\bullet$ & $\bullet$ \\ 
 \hline
 $v_1=v_2=v_3>0$  &  &  &  &  & \\ 
 \hline
 $v_1=v_2>v_3=0$& $\bullet$ & $\bullet$ &  &  &  \\ 
 \hline
 $v_1=v_2>v_3>0$ &  &  &  &  & \\ 
  \hline
 $v_1>v_2=v_3=0$& $\bullet$ & $\bullet$ & $\bullet$ &  &  \\ 
 \hline 
 $v_1>v_2=v_3>0$ & $\bullet$ &  &  &  & \\ 
 \hline
 $v_1>v_2>v_3=0$& $\bullet$ &  &  &  &  \\
 \hline 
 $v_1>v_2>v_3>0$ &  &  &  &  & \\  
\hline
\end{tabular}
        \end{center}
        \caption{$\lambda$-compatibility for $m=3$.}
        \label{tab:lambda_compatibility}
    \end{table}
\end{example}

\begin{lemma}\label{lem:affine_subspace}
Let $\sigma\in S_n$ have cycle type $\lambda=(\ell_1,\dots,\ell_m)$. Let $F$ be a forest on $[m]$ whose connected components induce the partition $\pi=\{B_1,\dots,B_k\}$ of $[m]$.
Then $\aff(\hobox_F)$ intersects the lattice $\Z^n$ if and only if $\pi$ is $\lambda$-compatible.
\end{lemma}

\begin{proof}
First we claim that 
\begin{equation} \label{eq:lambda comp}
\aff(\hobox_F) = \left\{\sum_{j=1}^m x_j \be_{\sigma_j} \, : \, 
                \sum_{j\in B_i}\ell_jx_j =
                \sum_{j\in B_i} \frac{\ell_j(\ell_j+1)}{2} \textrm{ for } 1\le i\le k \right\}.
\end{equation}
Let $E(F)$ be the edge set of $F$. We have
$
\aff(\hobox_F) =     \Span\{\ell_b\be_{\sigma_a}-\ell_a\be_{\sigma_b}:\{a,b\}\in E(F)\} + \sum_{a=1}^m\frac12({\ell_a+1})\be_{\sigma_a}.
$
A point $\by\in\Span\{\ell_b\be_{\sigma_a}-\ell_a\be_{\sigma_b}:\{a,b\}\in E(F)\}$ will satisfy $\sum_{j\in B_i}\ell_j y_j=0$ for each block $B_i$.
Furthermore, the translating vector $\bv:= \sum_{a=1}^m\frac12(\ell_a+1)\be_{\sigma_a}$ satisfies $\sum_{j\in B_i}\ell_j v_j = \sum_{j\in B_i}\frac12{\ell_j(\ell_j+1)}$ for each block $B_i$.
Thus every point $\bx$ in the affine span of $\hobox_F$ satisfies the given equations. 
These are all the relations among the $x_j$s because each block $B_i$ contributes $|E(B_i)|=|B_i|-1$ to the dimension of the affine span of $\hobox_F$.

This affine subspace intersects the lattice $\Z^n$ if and only if all equations in \eqref{eq:lambda comp} have integer solutions.
Elementary number theory tells us that this is the case if and only if each block $B_i$ satisfies
\begin{equation}\label{eq:div}
\gcd(\ell_j:j\in B_i) \;\Bigg|\; \sum_{j\in B_i} \frac{\ell_j(\ell_j+1)}{2}.
\end{equation}
It is always true that 
    $\gcd(\ell_j:j\in B_i)$ divides $\displaystyle \sum_{j\in B_i} \ell_j(\ell_j+1)$,
so \eqref{eq:div} holds if and only if
	\begin{equation}\label{eq:valuations}
	\val\Big(\gcd(\ell_j:j\in B_i)\Big) < \val\Big(\sum_{j\in B_i} \ell_j(\ell_j+1)\Big).
	\end{equation}
We consider two cases.

\medskip
\noindent
\emph{(i)}
Suppose $\ell_j$ is odd for some $j\in B_i$. Then $\gcd(\ell_j:j\in B_i)$ is odd, whereas $\sum_{j\in B_i}\ell_j(\ell_j+1)$ is always even.
Hence \eqref{eq:valuations} always holds in this case.

\medskip
\noindent
\emph{(ii)} 
Suppose that $\ell_j$ is even for all $j\in B_i$.
For each $\ell_j$, write $\ell_j = 2^{p_j}q_j$ for some integer $p_j\ge 1$ and odd integer $q_j$.
Then $\val(\gcd(\ell_j:j\in B_i)) = \min_{j\in B_i}p_j$; we will call this integer $p$.
We have
    \begin{align*}
        \val\Big(\sum_{j\in B_i}\ell_j(\ell_j+1)\Big)
        &= \val\Big(\sum_{j\in B_i}2^{p_j}q_j(\ell_j+1)\Big)\\
        &= p + \val\Big( \sum_{j\in B_i} 2^{p_j-p}q_j(\ell_j+1) \Big).
    \end{align*}
Note that $q_j(\ell_j+1)$ is odd for each $j$.
If the minimum $2$-valuation $p$ of $\{\ell_j \,: \, j\in B_i\}$ occurs an odd number of times, then $\sum_{j\in B_i}2^{p_j-p}q_j(\ell_j+1)$ will be odd and we will have $\val(\sum_{j\in B_i}\ell_j(\ell_j+1)) = p$.
Otherwise, this sum will be even and we will have $\val(\sum_{j\in B_i}\ell_j(\ell_j+1)) > p$.
Therefore \eqref{eq:valuations} holds if and only if the minimum $2$-valuation among the $\ell_j$ for $j\in B_i$ occurs an even number of times.
This is precisely the condition of $\lambda$-compatibility.
\end{proof}

We now have all of the necessary tools to compute the Ehrhart quasipolynomial of the fixed polytope $\Pi_n^\sigma$. Recall the definition of $\lambda$-compatibility in \Cref{def:compatibility} and the definition of $v_\pi$ in \eqref{eq:v_pi}.

\begin{reptheorem}{thm:main theorem}
Let $\sigma$ be a permutation of $[n]$ with cycle type $\lambda = (\ell_1, \ldots, \ell_m)$. 
Then the Ehrhart quasipolynomial of the fixed polytope of the permutahedron $\Pi_n$ fixed by $\sigma$ is
 \begin{equation*}
        L_{\Pi_n^\sigma}(t) = 
            \begin{cases}
                \displaystyle \quad \,\, \sum_{\pi \vDash[m]}v_\pi \cdot t^{m-|\pi|} &\text{if $t$ is even}
                \\
                \displaystyle\sum_{\substack{\pi \vDash[m]\\\lambda-\text{compatible}}} \hspace{-.5cm} v_\pi \cdot t^{m-|\pi|} &\text{if $t$ is odd}
            \end{cases}.
    \end{equation*}
\end{reptheorem}

\begin{proof}
We calculate the number of lattice points in each integer dilate $t\Pi_n^\sigma$ by decomposing it into half-open parallelotopes and adding up the number of lattice points inside of each parallelotope.

First, suppose that $t$ is even.
Then $t\Pi_n^\sigma$ is a lattice polytope, all parallelotopes in the decomposition of $t\Pi_n^\sigma$ have vertices on the integer lattice, and each $i$-dimensional parallelotope $\hobox$ contains $\vol(\hobox)t^i$ lattice points \cite[Lemma 9.2]{BeckRobins}.
The parallelotopes correspond to linearly independent subsets of the vector configuration $\{\ell_i\be_{\sigma_j}-\ell_j\be_{\sigma_i}:1\le i<j\le m\}$, which is in bijection with forests on $[m]$. Following the reasoning used to prove \Cref{thm:ehrhart_poly_lattice}, we conclude that when $t$ is even,
    \[L_{\Pi_n^\sigma}(t) = \sum_{\pi\vDash[m]}v_\pi \cdot t^{m-|\pi|}.\]

Next, suppose $t$ is odd.
Then $t\Pi_n^\sigma$ is half-integral, but it may not be a lattice polytope.
As before, we may decompose $t\Pi_n^\sigma$ into half-open parallelotopes that are in bijection with forests on $[m]$.
\Cref{lem:num_lattice_points}, \Cref{lem:affine_subspace}, and \cite[Lemma 9.2]{BeckRobins} tell us that $\hobox_F$ contains $\vol(\hobox_F)t^{m-|\pi|}$ lattice points if the set partition $\pi$ induced by $F$ is $\lambda$-compatible, and $0$ otherwise. Therefore if $t$ is odd
    $$
    L_{\Pi_n^\sigma}(t) = \sum_{\substack{\pi\vDash[m]\\\lambda-\text{compatible}}} v_\pi\cdot t^{m-|\pi|}
    $$
as desired.
\end{proof}

\section{The equivariant \texorpdfstring{$\phi$}{H*}-series of the permutahedron}
\label{sec:phi}

We now compute the equivariant $\phi$-series of the permutahedron and characterize when it is polynomial and when it is effective, proving Stapledon's Effectiveness \Cref{conj:Stapledon} in this special case.

The \emph{Ehrhart series} of a rational polytope $P$ is
    $$
    \Ehr_P(z) = 1 + \sum_{t=1}^\infty L_P(t)\cdot z^t.
    $$
In computing the Ehrhart series of $\Pi_n^\sigma$, Eulerian polynomials naturally arise.
The \emph{Eulerian polynomial} $A_k(z)$ is defined by the identity 
$$    \sum_{t\geq 0}t^kz^t = \dfrac{A_k(z)}{(1-z)^{k+1}}.
$$

\begin{proposition}\label{prop:Ehrhart_series}
Let $\sigma\in S_n$ have cycle type $\lambda=(\ell_1,\dots,\ell_m)$. 
The Ehrhart series of $\Pi_n^\sigma$ is
    \begin{equation*}
        \Ehr_{\Pi_n^\sigma}(z) = \sum_{\substack{\pi\vDash[m]\\\lambda\text{-compatible}}}\frac{v_\pi\cdot A_{m-|\pi|}(z)}{(1-z)^{m-|\pi|+1}}
        +
        \sum_{\substack{\pi\vDash[m] \\ \lambda\text{-incompatible}}} \frac{v_\pi\cdot 2^{m-|\pi|}\cdot A_{m-|\pi|}(z^2)}{(1-z^2)^{m-|\pi|+1}}
    \end{equation*}
and the $\phi$-series of the permutahedron equals
    \begin{equation*}
        \phi[z](\sigma) = \Big(\prod_{i=1}^m(1-z^{\ell_i})\Big)\cdot\Ehr_{\Pi_n^\sigma}(z)
    \end{equation*}
\end{proposition}

\begin{proof}
The first statement follows readily from \Cref{thm:main theorem}:
    \begin{align*}
     \Ehr_{\Pi_n^\sigma}(z) 
     &= \sum_{t \text{ even }}\left(\sum_{\pi\vDash [m]} v_\pi t^{m-|\pi|}\right)z^t + \sum_{t \text{ odd }}\left(\sum_{\substack{\pi\vDash[m]\\\lambda\text{-compatible}}} v_\pi t^{m-|\pi|}\right)z^t\\
     &= \sum_{\substack{\pi\vDash[m]\\\lambda\text{-compatible}}} v_\pi \left(\sum_{t=0}^\infty t^{m-|\pi|}z^t\right)+\sum_{\substack{\pi\vDash[m]\\ \lambda\text{-incompatible}}}v_\pi\left(\sum_{t \text{ even }}t^{m-|\pi|}z^t\right)\\
     &=\sum_{\substack{\pi\vDash[m]\\ \lambda\text{-compatible}}} v_\pi \frac{A_{m-|\pi|}(z)}{(1-z)^{m-|\pi|+1}} +\sum_{\substack{\pi\vDash[m] \\ \lambda\text{-incompatible}}} v_\pi\cdot 2^{m-|\pi|} \frac{A_{m-|\pi|}(z^2)}{(1-z^2)^{m-|\pi|+1}}
    \end{align*}
For the second statement, recall that $\phi[z]$ is defined as in \eqref{eq:phi_equation}, where $\rho$ is the standard representation of $S_n$ in this case. The left hand side is the Ehrhart series. The denominator on the right side is $(1-z)\det(I-\rho(\sigma)\cdot z)$; it equals the characteristic polynomial of the permutation matrix of $\sigma$, which is $\prod_{i=1}^m (1-z^{\ell_i})$. 
\end{proof}

Tables \ref{tab:n=3} and \ref{tab:n=4} show the equivariant $\phi$-series of the permutahedra $\Pi_3$ and $\Pi_4$.

    \begin{table}[ht]
    \centering
    \renewcommand*{\arraystretch}{2.5}
\begin{tabular}{|c|c|c|c|}
    \hline
        Cycle type of $\sigma\in S_3$ & $\chi_{t\Pi_3}(\sigma)$ & $\sum\limits_{t \geq 0} \chi_{t\Pi_3}(\sigma)z^t$ & $\phi[z](\sigma)$\\
   \hline
        $(1,1,1)$ & $3t^2 + 3t + 1$ & $\dfrac{1+4z+z^2}{(1-z)^3}$ & $1+4z+z^2$\\
    \hline
        $(2,1)$ & $\begin{cases}t+1 & \text{if $t$ is even}\\t & \text{if $t$ is odd} \end{cases}$ & $\dfrac{1+z^2}{(1-z)(1-z^2)}$ & $1+z^2$\\ 
    \hline
        $(3)$ & 1 & $\dfrac{1}{1-z} = \dfrac{1+z+z^2}{1-z^3}$ & $1+z+z^2$\\ 
    \hline
\end{tabular}
    \caption{The equivariant $\phi$-series of $\Pi_3$}
    \label{tab:n=3}
    \end{table}

    \begin{table}[ht]
    \resizebox{\textwidth}{!}{
    \renewcommand*{\arraystretch}{2.5}
\begin{tabular}{|c|c|c|c|}
    \hline
        Cycle type of $\sigma\in S_4$ & $\chi_{t\Pi_4}(\sigma)$ & $\sum\limits_{t \geq 0} \chi_{t\Pi_4}(\sigma)z^t$ & $\phi[z](\sigma)$ \\
    \hline
        $(1,1,1,1)$ & $16t^3 + 15t^2 + 6t +1$ & $\dfrac{1 + 34z + 55z^2 + 6z^3}{(1-z)^4}$ & $1 + 34z + 55z^2 + 6z^3$\\
   \hline
        $(2,1,1)$ & $\begin{cases} 4t^2 + 3t + 1 & \text{if $t$ is even}\\4t^2 + 2t & \text{if $t$ is odd} \end{cases}$ & $\dfrac{1 + 6z + 20z^2 + 24z^3 + 11z^4 + 2z^5}{(1-z)^2(1-z^2){\color{red}(1+z)^2}}$ & $1+4z+11z^2-2z^3+\displaystyle\sum_{i=4}^\infty 4(-1)^i z^i$ \\
    \hline
        $(3,1)$ &  $t+1$ & $\dfrac{1}{(1-z)^2} = \dfrac{1+z+z^2}{(1-z)(1-z^3)}$ & $1+z+z^2$\\
    \hline
        $(4)$ & $\begin{cases} 1 & \text{if $t$ is even}\\0 & \text{if $t$ is odd} \end{cases}$ & $\dfrac{1}{1-z^2}=\dfrac{1+z^2}{1-z^4}$ & $1+z^2$\\
    \hline
        $(2,2)$ &  $\begin{cases} 2t+1 & \text{if $t$ is even}\\ 2t & \text{if $t$ is odd} \end{cases}$ & $\dfrac{1+2z+3z^2+2z^3}{(1-z^2)^2}$ & $1+2z+3z^2+2z^3$\\
    \hline
\end{tabular}
    }
    \caption{The equivariant $\phi$-series of $\Pi_4$}
    \label{tab:n=4}
    \end{table}
    
Stapledon writes, \emph{``The main open problem is to characterize when $\phi[z]$ is effective"}, and he conjectures the following characterization:

\begin{repconjecture}{conj:Stapledon}[{\cite[Effectiveness Conjecture~12.1]{Stapledon}}]
Let $P$ be a lattice polytope invariant under the action of a group $G$.
The following conditions are equivalent.
	\begin{enumerate}[(i)]
	\item\label{item:toric} The toric variety of $P$ admits a $G$-invariant non-degenerate hypersurface.
	\item\label{item:effective} The equivariant $\phi$-series of $P$ is effective.
	\item\label{item:poly} The equivariant $\phi$-series of $P$ is a polynomial.
	\end{enumerate}
\end{repconjecture}

He shows that (\ref{item:toric}) $\implies$ (\ref{item:effective}) $\implies$ (\ref{item:poly}), so only the reverse implications are conjectured.
Our next goal is to verify Stapledon's conjecture for the action of $S_n$ on the permutahedron $\Pi_n$.
We do so by showing that the conditions of \Cref{conj:Stapledon} hold if and only if $n\le3$.

\subsection{Polynomiality of \texorpdfstring{$\phi[z]$}{H*[z]}}

\begin{lemma}\label{lem:phi[z](sigma) polynomial}
Let $\sigma\in S_n$ have cycle type $\lambda= (\ell_1,\dots,\ell_m)$.
The equivariant $\phi$-series evaluated at $\sigma$, $\phi[z](\sigma)$,  is a polynomial if and only if the number of even parts in $\lambda$ is $0$, $m-1$, or $m$.
\end{lemma}

\begin{proof}
By \Cref{prop:Ehrhart_series}, the Ehrhart series $\Ehr_{\Pi_n^\sigma}(z)$  may only have poles at $z=\pm1$.
The pole at $z=1$ has order at most $m$.
Since the polynomial $\prod_{i=1}^m(1-z^{\ell_i})$ has a zero at $z=1$ of order $m$, the series $\phi[z](\sigma)$ will not have a pole at $z=1$.
Hence we only need to check whether $\phi[z](\sigma)$ has a pole at $z=-1$.

\medskip

\noindent \emph{(i)} First, suppose no $\ell_i$ is even.
Then all partitions of $[m]$ are $\lambda$-compatible, so $\Ehr_{\Pi_n^\sigma}(z)$ does not have a pole at $z=-1$.
Thus $\phi[z](\sigma)$ is a polynomial in this case.

\medskip

\noindent \emph{(ii)}
Next, suppose that some $\ell_i$ is even.
Then the partition $\{\{\ell_i\},[m]-\{\ell_i\}\}$ is $\lambda$-incompatible, so $\Ehr_{\Pi_n^\sigma}(z)$ does have a pole at $z=-1$.
It is well known that $A_{k}(1) = k!$ so 
every numerator  $v_\pi \cdot 2^{m - |\pi|} \cdot A_{m-|\pi|}(z^2)$ is positive at $z=-1$.
It follows that the order of the pole $z=-1$ of $\Ehr_{\Pi_n^\sigma}(z)$ is $m-d+1$ where $d=\min\{|\pi|:\pi\text{ is }\lambda\text{-incompatible}\}$.
This equals $m-1$ if the partition $\{[m]\}$ is $\lambda$-compatible and $m$ if it is $\lambda$-incompatible.

On the other hand, $\prod_{i=1}^m(1-z^{\ell_i})$ has a zero at $z=-1$ of order equal to the number of even $\ell_i$. Now consider three cases:

\smallskip
\noindent
\emph{a)} If the number of even $\ell_i$ is between $1$ and $m-2$, 
it is less than the the order of the pole of $\Ehr_{\Pi_n^\sigma}(z)$, so $\phi[z](\sigma)$ is not a polynomial.

\smallskip
\noindent
\emph{b)} If all $\ell_i$ are even, the zero at $z=-1$ in $\prod_{i=1}^m(1-z^{\ell_i})$ 
has order $m$ and cancels the pole in $\Ehr_{\Pi_n^\sigma}(z)$. Thus $\phi[z](\sigma)$ is a polynomial.

\smallskip
\noindent
\emph{c)} If $m-1$ of the $\ell_i$ are even, the partition $\{[m]\}$ is $\lambda$-compatible. Therefore the order of the pole in $\Ehr_{\Pi_n^\sigma}(z)$ and the order of the zero in $\prod_{i=1}^m(1-z^{\ell_i})$ both equal $m-1$, and $\phi[z](\sigma)$ is a polynomial.
\end{proof}

\begin{proposition} \label{prop:poly}
The equivariant $\phi$-series of  the permutahedron $\Pi_n$ is a polynomial if and only if $n\le 3$.
\end{proposition}

\begin{proof}
When $n\le 3$, all partitions of $n$ have $0$, $1$, or all odd parts.
Hence $\phi[z](\sigma)$ is a polynomial for all $\sigma\in S_n$, so $\phi[z]$ is a polynomial.

Suppose $n\ge4$.
Then there always exists some partition of $n$ with more than $1$ but fewer than all odd parts: if $n$ is even we can take the partition $(n-2,1,1)$, and if $n$ is odd we can take the partition $(n-3,1,1,1)$.
Therefore $\phi[z]$ is not a polynomial.
\end{proof}

\subsection{Effectiveness of \texorpdfstring{$\phi[z]$}{H*[z]}}

\begin{proposition} \label{prop:eff}
The equivariant $\phi$-series of  the permutahedron $\Pi_n$ is effective if and only if $n\le 3$.
\end{proposition}

\begin{proof}
Stapledon \cite{Stapledon} observed that if $\phi$ is effective then it is a polynomial.
Thus by \Cref{prop:poly} we only need to check effectiveness for $n=1,2,3$. 

Let us check it for $n=3$. \Cref{tab:n=3} shows that $\phi[z]=\phi_0+\phi_1z+\phi_2z^2$ for $\phi_0, \phi_1, \phi_2 \in R(S_3)$. Comparing these with the character table of $S_3$ (see for example \cite[pg.14]{FultonHarris}) gives
\[
\phi_0=\chi_{triv},  \qquad \phi_1 = \chi_{triv}+\chi_{alt}+\chi_{std}, \qquad \phi_2 = \chi_{triv}.
\]
Since all coefficients are nonnegative, 
 $\phi_{\Pi_3}[z] = \chi_{triv} + (\chi_{triv}+\chi_{alt}+\chi_{std})z+\chi_{triv}z^2$
is indeed effective.

Similarly,  $\phi_{\Pi_2}[z] = \chi_{triv}$ and  $\phi_{\Pi_1}[z] = \chi_{triv}$ are effective as well.
\end{proof}

In contrast, a similar computation based on \Cref{tab:n=4} gives 
\begin{eqnarray*}
\phi_{\Pi_4} &=&
\chi_{triv} +
(3\chi_{triv} + \chi_{alt} + 5\chi_{std} + 3\chi_{\ysub{2,1,1}} + 3\chi_{\ysub{2,2}}) z + 
(6\chi_{triv} + 9\chi_{std} + 4\chi_{\ysub{2,1,1}} + 5\chi_{\ysub{2,2}})z^2 \\
&+& (\chi_{alt} + \chi_{\ysub{2,1,1}} + \chi_{\ysub{2,2}})z^3
+ (\chi_{triv}-\chi_{alt}+\chi_{std}-\chi_{\ysub{2,1,1}}) (z^4-z^5+z^6-z^7+\cdots)
\end{eqnarray*}
which is not effective.

\subsection{\texorpdfstring{$S_n$}{S\_n}-invariant non-degenerate hypersurfaces in the permutahedral variety}

We begin by explaining condition (\ref{item:toric}) of \Cref{conj:Stapledon}, which arises from Khovanskii's notion of non-degeneracy \cite{Khovanskii}. We refer the reader to \cite[Section 7]{Stapledon} for more details.

Let $P \subset \R^n$ be a lattice polytope with an action of a finite group $G$.
For $\bv\in \Z^n$ we write $x^\bv:=x_1^{v_1}\cdot\ldots\cdot x_n^{v_n}$.
The coordinate ring of the projective toric variety $X_P$  of $P$ has the form $\C[x^\bv:\bv\in P\cap \Z^n]$, so a hypersurface in $X_P$ is given by a linear equation $\sum_{\bv\in P\cap \Z^n}a_\bv x^\bv=0$ for some complex coefficients $a_\bv$.
The group $G$ acts on the monomials $x^\bv$ by its action on the lattice points $\bv\in P\cap \Z^n$, so the equation of a $G$-invariant hypersurface should have $a_\bv = a_\bu$ whenever $\bu$ and $\bv$ are in the same $G$-orbit.
A projective hypersurface in $X_P$ with equation $f(x_1,\dots,x_n)=0$ is \emph{smooth} if the gradient $(\partial f/\partial x_1,\dots,\partial f/\partial x_n)$ is never zero when $(x_1,\dots,x_n)\in(\C^*)^n$.
There is a unique polynomial in the $a_\bv$'s, called the \emph{discriminant}, such that  the hypersurface is smooth when the discriminant does not vanish at the coefficients $a_\bv$. A hypersurface in the toric variety of $P$ is \emph{non-degenerate} if it is smooth and for each face $F$ of $P$, the hypersurface $\sum_{\bv\in F\cap \Z^n}a_\bv x^\bv=0$ is also smooth.

The \emph{permutahedral variety} $X_{\Pi_n}$ is the projective toric variety associated to the permutahedron $\Pi_n$.

\begin{proposition}\label{prop:smooth}
The permutahedral variety $X_{\Pi_n}$ admits an $S_n$-invariant non-degenerate hypersurface if and only if $n \leq 3$.
\end{proposition}

\begin{proof}
Stapledon proved \cite[Theorem 7.7]{Stapledonmirror} that if $X_{\Pi_n}$ admits such a hypersurface, then $\phi[z]$ is effective. By \Cref{prop:eff}, this can only occur for $n=1,2,3$.

\medskip
\noindent Case 1: $n=1$.

A hypersurface in the toric variety of $\Pi_1=\{1\}\subset\R$ has the form $ax=0$, and since we are working over projective space, we can assume $a=1$.
The derivative of this never vanishes, so this is a smooth $S_1$-invariant hypersurface.

\medskip
\noindent Case 2: $n=2$.

The permutahedron $\Pi_2$ is the line segment with vertices $(1,2), (2,1)\in\R^2$ and no other lattice points.
The vertices are in the same $S_2$-orbit, so we need to check that hypersurface with equation $xy^2+x^2y=0$ is non-degenerate.
The gradient is $(y(y+2x),x(2y+x))$, which never vanishes on $(\C^*)^2$.
The vertex $(1,2)$ corresponds to the hypersurface $xy^2=0$.
The gradient of this is $(y^2,2xy)$ which also never vanishes on $(\C^*)^2$.
The computation for the other vertex is similar.
Hence this is an $S_2$-invariant non-degenerate hypersurface.

\medskip
\noindent Case 3: $n=3$.

The permutahedron $\Pi_3$ is a hexagon with one interior point.
Choosing the vertices to be all permutations of the point $(0,1,2)\in\R^3$ (instead of $(1,2,3)$) will simplify calculations.
The six vertices of the hexagon are one $S_3$-orbit and the interior point is its own orbit.
Hence (up to scaling) an $S_3$-invariant hypersurface must have the equation
    \begin{equation}\label{eq:hypersurface}
        a\cdot xyz + yz^2+ y^2z+xy^2+x^2y+xz^2+x^2z=0
    \end{equation}
which has one parameter $a$.
We want to check whether there exists some choice of $a$ for which this hypersurface is non-degenerate. We need to check this on each face.

The vertex $(0,1,2)$ gives the hypersurface $yz^2=0$ with gradient $(0,z^2,2yz)$.
This never vanishes on $(\C^*)^3$, so it is smooth.
The computations for the other five vertices are similar.

For the edge connecting $(0,1,2)$ and $(0,2,1)$, the corresponding hypersurface is $yz^2 + y^2z = 0$. This is the same hypersurface as the line segment $\Pi_2$, so it is smooth; so are the hypersurfaces of the other five edges.

Finally, we need to show there exists $a$ such that the entire hypersurface is smooth. This is the same as showing that the discriminant of \eqref{eq:hypersurface} is not identically zero.
Since \eqref{eq:hypersurface} is a symmetric polynomial, we can write in terms of the power-sum symmetric polynomials, $p_k = x^k+y^k+z^k$; we obtain
\begin{equation}\label{eq:powersum}
        \frac{a}{6}p_1^3 + \left(1-\frac{a}{2}\right)p_1p_2 + \left(\frac{a}{3}-1\right)p_3=0.
\end{equation}
The discriminant of a degree $3$ symmetric polynomial is given in  \cite[Equation 64]{discriminant}; substituting the coefficients $a/6$, $1-a/2$, and $a/3-1$ gives a non-zero polynomial of degree $12$:
	\begin{eqnarray*}
	&\frac{-512000}{16677181699666569}a^{12} + \frac{492800}{617673396283947}a^{10} - \frac{985600}{617673396283947}a^9 + \frac{6320}{7625597484987}a^8 \\
	& -\frac{25280}{7625597484987}a^7 + \frac{27431}{7625597484987}a^6 - \frac{478}{282429536481}a^5 + \frac{965}{282429536481}a^4 \\
	& -\frac{2128}{847288609443}a^3 + \frac{8}{10460353203}a^2 - \frac{32}{31381059609}a + \frac{16}{31381059609}
	\end{eqnarray*}
Any value of $a$ that is not a root of this discriminant gives us an $S_3$-invariant non-degenerate hypersurface.
\end{proof}

By contrast, we should not be able to find an $S_n$-invariant non-degenerate hypersurface in $X_{\Pi_n}$ for $n \ge 4$. 
This can be seen from the fact that all permutahedra $\Pi_n$ when $n\ge 4$ have a square face, and the hypersurface of this square face is not smooth.
For example, consider the square face of $\Pi_4$ with vertices $(0,1,2,3)$, $(0,1,3,2)$, $(1,0,3,2)$, and $(1,0,2,3)$.
The corresponding hypersurface is $yz^2w^3 + yz^3w^2+xz^3w^2+xz^2w^3=0$, and its gradient vanishes whenever $x=-y$ and $z=-w$.

\subsection{Stapledon's Conjectures}

Our second main result now follows.

\begin{reptheorem}{thm:phi polynomial}
Stapledon's Effectiveness Conjecture holds for the permutahedron under the action of the symmetric group.
\end{reptheorem}

\begin{proof}
This follows immediately from Propositions \ref{prop:poly}, \ref{prop:eff}, and \ref{prop:smooth}.
\end{proof}

In closing, we verify the remaining three conjectures of Stapledon for the special case of the $S_n$--action on the permutahedron $\Pi_n$.

\begin{conjecture}\cite[Conjecture~12.2]{Stapledon}\label{conj:12.2}
If $\phi[z]$ is effective, then $\phi[1]$ is a permutation representation.
\end{conjecture}

\begin{conjecture}\cite[Conjecture~12.3]{Stapledon}\label{conj:12.3}
For a polytope $P\subset\R^n$, let $\mathrm{ind}(P)$ be the smallest positive integer $k$ such that the affine span of $kP$ contains a lattice point.
For any $g\in G$, let $M^g$ be the sublattice of $M$ fixed by $g$, and define $\det(I-\rho(g))_{(M^g)^\perp}$ to be the determinant of $I-\rho(g)$ when the action of $\rho(g)$ is restricted to $(M^g)^\perp$.
The quantity 
	\[\phi[1](g) = \frac{\dim(P^g)!\cdot\vol(P^g)\cdot\det(I-\rho(g))_{(M^g)^\perp}}{\mathrm{ind}(P^g)}\]
is a non-negative integer.
\end{conjecture}

\begin{conjecture}\cite[Conjecture~12.4]{Stapledon}\label{conj:12.4}
If $\phi[z]$ is a polynomial and the $i^{\text{th}}$ coefficient of the $h^*$-polynomial of $P$ is positive, then the trivial representation occurs with non-zero multiplicity in the virtual character $\phi_i$.
\end{conjecture}

\begin{proposition}\label{prop:conjectures hold}
Conjectures~\ref{conj:12.2}, \ref{conj:12.3}, and \ref{conj:12.4} hold for permutahedra under the action of the symmetric group.
\end{proposition}

\begin{proof}
\ref{conj:12.2}:
This statement only applies to $\Pi_1$, $\Pi_2$, and $\Pi_3$. From the proof of \Cref{prop:eff} we obtain that $\phi[1]$ is the trivial character for $\Pi_1$ and $\Pi_2$ and the statement holds.
For $\Pi_3$ we have 
	\begin{equation}\label{eq:permchar}
	\phi[1]=3\chi_{triv} + \chi_{alt} + \chi_{std} = \chi_{triv} + (\chi_{triv}+\chi_{alt}) + (\chi_{triv}+\chi_{std}).
	\end{equation}
Now $\chi_{triv}+\chi_{alt}$ is the permutation character of the sign action of $S_3$ on the set $[2]$, and $\chi_{triv}+\chi_{std}$ is the character of the permutation representation of $S_3$.
Hence all summands on the right side of \eqref{eq:permchar} are permutation characters, so their sum is as well.

\ref{conj:12.3}:
For $\sigma\in S_n$ of cycle type $\lambda=(\ell_1,\dots,\ell_m)$, the dimension of $\Pi_n^\sigma$ is $m-1$ and the volume is $n^{m-2}\gcd(\ell_1,\dots,\ell_m)$.
Now, the fixed lattice $M^g= \Z\{\be_{\sigma_1}, \cdots, \be_{\sigma_m}\}$ has rank $m$, so
	\[ 
	\det(I-\rho(\sigma)\cdot z)_{(M^\sigma)^\perp} =
	\frac{(1-z)\det(I-\rho(\sigma)\cdot z)}{(1-z)^m} =
	 \prod_{i=1}^m (1+z+\dots+z^{\ell_i-1}).
	 \]
	 Therefore the numerator is $(m-1)!\cdot n^{m-2} \cdot \gcd(\ell_1,\dots,\ell_m) \cdot \ell_1\cdots\ell_m$. The denominator is
	\[\mathrm{ind}(\Pi_n^\sigma) =
	\begin{cases} 
	2 &\text{if all } \pi\vDash[m] \text{ are }\lambda\text{-incompatible,}\\ 
	1 &\text{otherwise}.
	 \end{cases}
	  \]
When the denominator is $2$, all the $\ell_i$ must be even, so the numerator is even. The desired result follows.

\ref{conj:12.4}:
We need to check this for $\Pi_1$, $\Pi_2$, and $\Pi_3$.
For $\Pi_1$ and $\Pi_2$ the $h^*$-polynomial is $1$ and $\phi_0 =\chi_{triv}$.
For $\Pi_3$, the $h^*$-polynomial is $1+4z+z^2$, and $\phi_0  = \chi_{triv} $, $\phi_1 = \chi_{triv} + \chi_{alt} + \chi_{std}$, and $\phi_2 = \chi_{triv} $ all contain a copy of the trivial character.
\end{proof}

\section{Acknowledgments}
We would like to thank Sophia Elia for developing a useful Sage package for equivariant Ehrhart theory, and Matthias Beck, Benjamin Braun, Christopher Borger, Ana Botero, Sophia Elia, Donghyun Kim, Jodi McWhirter, Dusty Ross, Kristin Shaw, and Anna Schindler for fruitful conversations.
This work was completed while FA was a Spring 2019 Visiting Professor at the Simons Institute for Theoretical Computer Science in Berkeley, and a 2019-2020 Simons Fellow while on sabbatical in Bogot\'a; he is very grateful to the Simons Foundation, San Francisco State University, and the Universidad de Los Andes  for their support.

\small
\bibliographystyle{amsplain}
\bibliography{references}

\end{document}